\documentclass{amsart}
\usepackage{bbding}
\usepackage{amssymb}
\usepackage{amsmath}

\newtheorem{theorem}{Theorem}[section]

\newtheorem{proposition}[theorem]{Proposition}
\newtheorem{corollary}[theorem]{Corollary}
\theoremstyle{definition}
\newtheorem{definition}[theorem]{Definition}
\newtheorem{example}[theorem]{Example}

\newtheorem{remark}[theorem]{Remark}

\newcommand{\blankbox}[2]

\begin{document}
\title{Central extensions and conformal derivations of a class of Lie conformal algebras}

\author{Yanyong Hong}
\address{Department of Mathematics, Hangzhou Normal University,
Hangzhou, 311121, P.R.China}
\email{hongyanyong2008@yahoo.com}

\thanks{Project supported by the Scientific Research Foundation of Hangzhou Normal University (No. 2019QDL012) and the National Natural Science Foundation of China (No.11501515, 11871421)}
\begin{abstract}
A quadratic Lie conformal algebra corresponds to a Hamiltonian pair in \cite{GD}, which plays fundamental roles in completely integrable systems. Moreover, it also corresponds to certain compatible pairs of a Lie algebra and a Novikov algebra which was called Gel'fand-Dorfman bialgebra by Xu in \cite{X1}.
In this paper, central extensions and conformal derivations of quadratic Lie conformal algebras are studied in terms of Gel'fand-Dorfman bialgebras. It is shown that central extensions and conformal derivations of a quadratic Lie conformal algebra are related with some bilinear forms and some operators of the corresponding Gel'fand-Dorfman bialgebra respectively.

\end{abstract}

\maketitle

\section{Introduction}
The notion of vertex algebra was formulated by R.Borcherds
in \cite{Bo}, which is a rigorous mathematical definition of the
chiral part of a 2-dimensional quantum field theory studied
intensively by physicists since the landmark paper \cite{BPZ}. Lie
conformal algebra introduced by V.Kac in \cite{K1,K2} is an
axiomatic description of the operator product expansion (or rather
its Fourier transform) of chiral fields in conformal field theory.
In a way, its relationship with vertex algebra is like the relationship between Lie algebra and  its universal enveloping algebra.
It turns out to be a valuable tool in studying vertex algebras.
In addition, Lie conformal algebras have close connections to Hamiltonian formalism in the theory of nonlinear evolution equations (see the book \cite{Do} and references therein, and also \cite{BDK, GD, Z, X4} and many other papers).

Moreover, in the purely algebraic viewpoint, conformal algebras
are quite intriguing subjects. One can define the conformal analogue
of a variety of ``usual" algebras such as Lie conformal algebras,
associative conformal algebras, etc. The theory of some variety of
conformal algebras sheds new light on the problem of classification
of infinite-dimensional algebras of the corresponding ``classical"
variety. In particular, Lie conformal algebras give us powerful tools for the study of infinite-dimensional Lie algebras satisfying the locality property in \cite{K}.

The topic of this paper is about Lie conformal algebra. Simple (or semisimple) Lie conformal algebras have been intensively investigated.
A complete classification of finite simple (or semisimple) Lie conformal algebras was given in \cite{DK1},
all finite irreducible conformal modules of finite simple (or semisimple) Lie conformal algebras were classified in \cite{CK1}, extensions of these conformal modules were studied in \cite{CK2}
and cohomology groups of finite simple Lie conformal algebras with some conformal modules
were characterized in \cite{BKV}. Recently, some non-simple Lie conformal algebras were introduced and their structures including central extensions, conformal derivations were studied. For example, two new nonsimple Lie conformal algebras were studied in \cite{SY} which are obtained from two formal distribution Lie algebras,
i.e. Schr\"{o}dinger-Virasoro Lie algebra and the extended Schr\"{o}dinger-Virasoro Lie algebra. Similarly, the Lie conformal algebras of the loop Virasoro Lie algebra and a Block type Lie algebra were studied in \cite{WCY} and \cite{GXY} respectively. In fact, all these Lie conformal algebras including semisimple Lie conformal algebras are quadratic Lie conformal algebras named by Xu in \cite{X1}.
A quadratic Lie conformal algebra corresponds to a Hamiltonian pair in \cite{GD}, which plays fundamental roles in completely integrable systems. Moreover, it also corresponds to certain compatible pairs of a Lie algebra and a Novikov algebra which was called Gel'fand-Dorfman bialgebra by Xu in \cite{X1}.

As we know, the description of finite Lie conformal algebras splits into two problems (see \cite{K1}):\\
(1) describe Lie conformal algebras that are free $\mathbb{C}[\partial]$-modules of finite rank;\\
(2) find central extensions of Lie conformal algebras from (1) with center being in torsion.\\
Therefore, it is meaningful to study central extensions of  Lie conformal algebras in the structure theory of Lie conformal algebras. Since most of known Lie conformal algebras are quadratic Lie conformal algebras, in this paper, we plan to study conformal derivations and central extensions of quadratic Lie conformal algebras
by a one-dimensional center $\mathbb{C}\mathfrak{c}_\beta$ where $\partial\mathfrak{c}_\beta =\beta \mathfrak{c}_\beta $ in a unified form. It is shown that central extensions of some quadratic Lie conformal algebras including those corresponding to the simple Novikov algebras or Novikov algebras with a left unit or a right unit  by a one dimensional center  $\mathbb{C}\mathfrak{c}_\beta$ can be directly obtained from some bilinear forms of the corresponding Gel'fand-Dorfman bialgebras satisfying several equalities. Note that this method can be also applied to study central extensions of quadratic Lie conformal algebras
by an abelian Lie conformal algebra $\mathbb{C}[\partial]\mathfrak{c}$ which is free of rank one. In particular, we prove that $H^2(R, \mathbb{C}\mathfrak{c}_\beta)=0$ and $H^2(R, \mathbb{C}[\partial]\mathfrak{c})=0$ when $\beta\neq 0$ and $R$ is the quadratic Lie conformal algebra corresponding to Novikov algebra with a right unit. In addition, we show that conformal derivations of some quadratic Lie conformal algebras including those corresponding to the Novikov algebras with a left unit or with a right unit can be obtained from some operators of the corresponding Gel'fand-Dorfman bialgebras. Moreover, it is proved that all conformal derivations of quadratic Lie conformal algebra corresponding to the Novikov algebra with a unit are inner.

This paper is organized as follows. In Section 2, the definitions of Lie conformal algebra and quadratic Lie conformal algebra are recalled.
In Section 3, central extensions of quadratic Lie conformal algebras by a one-dimensional center $\mathbb{C}\mathfrak{c}_\beta$ where $\partial\mathfrak{c}_\beta =\beta \mathfrak{c}_\beta $ are studied in terms of Gel'fand-Dorfman bialgebras. It is shown that $H^2(R, \mathbb{C}\mathfrak{c}_\beta)=0$ and $H^2(R, \mathbb{C}[\partial])=0$ when $\beta\neq 0$ and $R$ is the quadratic Lie conformal algebra corresponding to Novikov algebra with a right unit. As an application, we characterize central extensions of several specific quadratic Lie conformal algebras by $\mathbb{C}\mathfrak{c}_0$ up to equivalence. In Section 4, we study conformal derivations of quadratic Lie conformal algebras in terms of Gel'fand-Dorfman bialgebras. Moreover,
conformal derivations of several specific quadratic Lie conformal algebras are determined.

Throughout this paper, denote by $\mathbb{C}$ the field of complex
numbers; $\mathbb{N}$ is the set of natural numbers, i.e.
$\mathbb{N}=\{0, 1, 2,\cdots\}$; $\mathbb{Z}$ is the set of integer
numbers; denote by $C_j^i$ the corresponding binomial coefficient when $i$, $j\in \mathbb{N}$ and $j\geq i$. All tensors over $\mathbb{C}$ are denoted by $\otimes$.
Moreover, if $A$ is a vector space, the space of polynomials of $\lambda$ with coefficients in $A$ is denoted by $A[\lambda]$.

\section{Preliminaries}
In this section, we will recall the definitions of Lie conformal algebra and quadratic Lie conformal algebra. These facts can be referred to \cite{K1} and \cite{X1}.
\begin{definition}
A \emph{Lie conformal algebra} $R$ is a $\mathbb{C}[\partial]$-module with a $\lambda$-bracket $[\cdot_\lambda \cdot]$ which defines a $\mathbb{C}$-bilinear
map from $R\times R\rightarrow R[\lambda]$, satisfying
\begin{eqnarray*}
&&[\partial a_\lambda b]=-\lambda [a_\lambda b],~~~[a_\lambda \partial b]=(\lambda+\partial)[a_\lambda b], ~~\text{(conformal sesquilinearity)}\\
&&[a_\lambda b]=-[b_{-\lambda-\partial}a],~~~~\text{(skew-symmetry)}\\
&&[a_\lambda[b_\mu c]]=[[a_\lambda b]_{\lambda+\mu} c]+[b_\mu[a_\lambda c]],~~~~~~\text{(Jacobi identity)}
\end{eqnarray*}
for $a$, $b$, $c\in R$.
\end{definition}
A Lie conformal algebra is said to be \emph{finite} if it is finitely generated  as a $\mathbb{C}[\partial]$-module. Otherwise, it is called \emph{infinite}.

In addition, there is an important Lie algebra associated with a Lie conformal algebra.
Assume that $R$ is a Lie conformal algebra and set $[a_\lambda b]=\sum_{n\in \mathbb{N}}\frac{\lambda^n}{n!}a_{(n)}b$ where $a_{(n)}b\in R$ for any $n\in \mathbb{N}$. Let Coeff$(R)$ be the quotient
of the vector space with basis $a_n$ $(a\in R, n\in\mathbb{Z})$ by
the subspace spanned over $\mathbb{C}$ by
elements:
$$(\alpha a)_n-\alpha a_n,~~(a+b)_n-a_n-b_n,~~(\partial
a)_n+na_{n-1},~~~\text{where}~~a,~~b\in R,~~\alpha\in \mathbb{C},~~n\in
\mathbb{Z}.$$ The operation on Coeff$(R)$ is defined as follows:
\begin{equation}\label{106}
[a_m, b_n]=\sum_{j\in \mathbb{N}}\left(\begin{array}{ccc}
m\\j\end{array}\right)(a_{(j)}b)_{m+n-j}.\end{equation} Then
Coeff$(R)$ is a Lie algebra and it is called the\emph{ coefficient algebra} of $R$ (see \cite{K1}).

Next, let us introduce some examples of Lie conformal algebras.

\begin{example}
The Virasoro Lie conformal algebra $\text{Vir}$ is the simplest nontrivial
example of Lie conformal algebras. It is defined by
$$\text{Vir}=\mathbb{C}[\partial]L, ~~[L_\lambda L]=(\partial+2\lambda)L.$$
Coeff$\text{(Vir)}$ is just the Witt algebra.
\end{example}

\begin{example}

Let $\mathfrak{g}$ be a Lie algebra. The current Lie conformal
algebra associated to $\mathfrak{g}$ is defined by
$$\text{Cur} \mathfrak{g}=\mathbb{C}[\partial]\otimes \mathfrak{g}, ~~[a_\lambda b]=[a,b],
~~a,b\in \mathfrak{g}.$$

Moreover, we can define a semi-direct sum of $\text{Vir}$ and $\text{Cur}
\mathfrak{g}$. The $\mathbb{C}[\partial]$-module $\text{Vir}\oplus \text{Cur}
\mathfrak{g}$ can be given a Lie conformal algebra structure by
$$[L_\lambda L]=(\partial+2\lambda)L,~~~[a_\lambda b]=[a,b],~~[L_\lambda
a]=(\partial+\lambda)a,$$ $L$ being the standard generator of $\text{Vir}$,
$a$, $b\in \mathfrak{g}$.
\end{example}

Then we introduce a class of special Lie conformal algebras named quadratic Lie conformal algebras (see \cite{X1}).
\begin{definition}
$R$ is a \emph{quadratic Lie conformal algebra}, if there exists a vector space $V$ such that $R=\mathbb{C}[\partial]V$ is a Lie conformal algebra as a free
$\mathbb{C}[\partial]$-module and the $\lambda$-bracket is of the following form
$$[a_{\lambda} b]=\partial u+\lambda v+ w,$$
where $a$, $b$, $u$, $v$, $w\in V$.
\end{definition}
\begin{remark}
Obviously, $\text{Vir}$, $\text{Cur}\mathfrak{g}$ and $\text{Vir}\oplus \text{Cur}
\mathfrak{g}$ where $\mathfrak{g}$ is a Lie algebra are quadratic Lie conformal algebras.
Moreover, all Lie conformal algebras introduced in \cite{SY, WCY, GXY} are quadratic Lie conformal algebras.
\end{remark}

For giving an equivalent characterization of quadratic Lie conformal algebras, we first introduce the definitions of Novikov algebra and
Gel'fand-Dorfman bialgebra.
\begin{definition}
A \emph{Novikov algebra} $V$ is a vector space over $\mathbb{C}$ with a bilinear product $\circ: V\times V\rightarrow V$ satisfying (for any $a$, $b$, $c\in V$):
\begin{eqnarray*}
&&(a\circ b)\circ c-a \circ (b\circ c)=(b\circ a)\circ c-b \circ (a\circ c),\\
&&(a\circ b)\circ c=(a\circ c)\circ b.
\end{eqnarray*}

If for any $a\in V$, there exists $x\in V$ such that $x\circ a=a$ or $a\circ x=a$, then $x$ is called a \emph{left unit} or \emph{right unit} of $(V,\circ)$. If an element in $V$ is not only a left unit but also a right unit, then we call it a \emph{unit}. If $I$ is a subspace of a Novikov algebra $(V,\circ)$ and $V\circ I\subset I$, $I\circ V\subset I$, then $I$ is called an \emph{ideal} of $(V,\circ)$. Obviously, $0$ and $V$ are ideals of $(V,\circ)$, which are called \emph{trivial}. $(V,\circ)$ is called \emph{simple}, if
$(V,\circ)$ has only trivial ideals and $a\circ b\neq 0$ for some $a$, $b\in V$.

\begin{remark}
Novikov algebra was essentially stated in \cite{GD} that it corresponds to a
certain Hamiltonian operator. Such an algebraic structure also appeared
in \cite{BN} from the point of view of Poisson structures of
hydrodynamic type. The name ``Novikov algebra" was given by Osborn
in \cite{Os}.
\end{remark}
\end{definition}

\begin{definition}(see \cite{GD} or \cite{X1})
A \emph{Gel'fand-Dorfman bialgebra} $V$ is a vector space over $\mathbb{C}$ with two algebraic operations $[\cdot,\cdot]$ and $\circ$ such that $(V,[\cdot,\cdot])$ forms a Lie algebra, $(V,\circ)$ forms a Novikov algebra and the following compatibility condition holds:
\begin{eqnarray*}
[a\circ b, c]-[a\circ c, b]+[a,b]\circ c-[a,c]\circ b-a\circ [b,c]=0,
\end{eqnarray*}
for any $a$, $b$, and $c\in V$. We usually denote it by $(V,\circ,[\cdot,\cdot])$.
\end{definition}
\begin{remark}
By the definition of Gel'fand-Dorfman bialgebra, any Lie algebra with the trivial Novikov algebra structure and any Novikov algebra with the trivial Lie algebra structure are Gel'fand-Dorfman bialgebras.
\end{remark}

An equivalent characterization of quadratic Lie conformal algebra is presented as follows.
\begin{theorem}(see \cite{X1})
$R=\mathbb{C}[\partial]V$ is a quadratic Lie conformal algebra if and only if the $\lambda$-bracket of $R$ is given as
follows
$$[a_{\lambda} b]=\partial(b\circ a)+[b, a]+\lambda(b\circ a+a\circ b), \text{$a$, $b\in V$},$$
and $(V, \circ, [\cdot,\cdot])$ is a Gel'fand-Dorfman bialgebra. Therefore, $R$ is called the quadratic Lie conformal algebra corresponding to
the Gel'fand-Dorfman bialgebra $(V, \circ, [\cdot,\cdot])$.
\end{theorem}

Finally, we introduce some notations about the Gel'fand-Dorfman bialgebra $(V, \circ, [\cdot,\cdot])$.

Denote $a\ast b=a\circ b+b\circ a$ for any $a$, $b\in V$. Then $(V, \ast)$ is a commutative but not (usually) associative algebra.
Therefore, the $\lambda$-bracket of a quadratic Lie conformal algebra can be written as
\begin{eqnarray}\label{gf1}
[a_{\lambda} b]=\partial(b\circ a)+[b, a]+\lambda(a\ast b), \text{$a$, $b\in V$}.
\end{eqnarray}
For convenience, the coefficient algebra of the Lie conformal algebra corresponding to $(V, \circ, [\cdot,\cdot])$ is
denoted by  $\mathcal{L}(V)$.

Moreover, if a Gel'fand-Dorfman bialgebra is a Novikov algebra with the trivial Lie algebra structure, for convenience, then we usually say the quadratic Lie conformal algebra corresponds to the Novikov algebra.

\section{Central extensions of quadratic Lie conformal algebras}

In this section, we will study central extensions of quadratic Lie conformal algebras by a one-dimensional center $\mathbb{C}\mathfrak{c}_\beta$ with $\partial \mathfrak{c}_\beta=\beta \mathfrak{c}_\beta$. We denote $\mathfrak{c}_0$ by $\mathfrak{c}$. The method using here can be also applied to central extensions of quadratic Lie conformal algebras by an
abelian Lie conformal algebra $\mathbb{C}[\partial]\mathfrak{c}$ which is free of rank one as a $\mathbb{C}[\partial]$-module.

An extension of a Lie conformal algebra $R$ by an abelian Lie conformal algebra $C$ is a short exact sequence of Lie conformal algebras
\begin{eqnarray*}
0\rightarrow C\rightarrow \widehat{R}\rightarrow R\rightarrow 0.
\end{eqnarray*}
$\widehat{R}$ is called an \emph{extension} of $R$ by $C$ in this case. This extension is called \emph{central} if
$[C_\lambda \widehat{R}]=0$.

In the following, we focus on the central extension $\widehat{R}$ of $R$ by a one-dimensional center $\mathbb{C}\mathfrak{c}_\beta$. Since quadratic Lie conformal algebras are free as a $\mathbb{C}[\partial]$-module, we always assume that $R$ is free as a $\mathbb{C}[\partial]$-module. This implies that $\widehat{R}=R\oplus \mathbb{C}\mathfrak{c}_\beta$, and
\begin{eqnarray*}
[a_\lambda b]_{\widehat{R}}=[a_\lambda b]_R+\alpha_\lambda(a,b)\mathfrak{c}_\beta, \text{for~~all~~$a$, $b\in R$,}
\end{eqnarray*}
where $\alpha_\lambda(\cdot,\cdot): R\times R\rightarrow \mathbb{C}[\lambda]$ is a $\mathbb{C}$-bilinear map. By the axioms of Lie conformal algebra, $\alpha_\lambda(\cdot,\cdot)$ should satisfy the following properties (for all $a$, $b$, $c\in R$) :
\begin{eqnarray}
&&\label{e3}\alpha_\lambda(\partial a,b)=-\lambda \alpha_\lambda(a,b),~~~\alpha_\lambda(a,\partial b)=(\lambda+\beta)\alpha_\lambda(a,b),\\
&&\label{e4}\alpha_\lambda(a,b)=-\alpha_{-\lambda-\beta}(b,a),\\
&&\label{ec1}\alpha_\lambda(a,[b_\mu c])-\alpha_\mu(b,[a_\lambda c])=\alpha_{\lambda+\mu}([a_\lambda b],c).
\end{eqnarray}
According to the cohomology theory of Lie conformal algebra in \cite{BKV}, $\alpha_\lambda(\cdot,\cdot)$ is a 2-cocycle in the reduced complex of $R$ with values in the trivial $R$-module $\mathbb{C}_\beta$ where $\mathbb{C}_\beta=\mathbb{C}$ and $\partial k=\beta k$ for any $k\in \mathbb{C}_\beta$. Denote $\mathbb{C}_0$ by $\mathbb{C}$. It is shown in \cite{BKV} that the central extensions of $R$ by $\mathbb{C}\mathfrak{c}_\beta$ up to equivalence can be parameterized by $H^2(R,\mathbb{C}\mathfrak{c}_\beta)$. Moreover, by the definition of 2-coboundary, 2-cocycles $\alpha_\lambda(\cdot,\cdot)$ and $\alpha_\lambda^{'}(\cdot,\cdot)$ are equivalent if and only if there exists a $\mathbb{C}[\partial]$-module homomorphism $\varphi: R\rightarrow \mathbb{C}_\beta$  such that $\alpha_\lambda(a,b)=\alpha_\lambda^{'}(a,b)+\varphi([a_\lambda b])$ for all
$a$, $b\in R$. Note that equivalent 2-cocycles define equivalent central extensions.

Central extensions of quadratic Lie conformal algebras by a one-dimensional center $\mathbb{C}\mathfrak{c}_\beta$ are characterized as follows.
\begin{theorem}\label{t1}
Let $\widehat{R}=R\oplus\mathbb{C}\mathfrak{c}_\beta$ be a central extension of quadratic Lie conformal algebra $R=\mathbb{C}[\partial]V$  corresponding to $(V,\circ, [\cdot,\cdot])$ by a one-dimensional center $\mathbb{C}\mathfrak{c}_\beta$.
Set the $\lambda$-bracket of $\widehat{R}$ by
\begin{eqnarray}
\label{e2}\widetilde{[a_\lambda b]}=\partial(b\circ a)+\lambda(a\ast b)+[b,a]+\alpha_\lambda(a,b)\mathfrak{c}_\beta,
\end{eqnarray}
where $a$, $b\in V$ and $\alpha_\lambda(a,b)\in \mathbb{C}[\lambda]$. Assume that $\alpha_\lambda(a,b)=\sum_{i=0}^n\lambda^i\alpha_i(a,b)$ for any $a$, $b\in V$, where $\alpha_i(\cdot,\cdot):V\times V\rightarrow \mathbb{C}$ are bilinear forms on $V$ and there exist some $a$, $b \in V$ such that
$\alpha_n(a,b)\neq 0$. Then we obtain (for any $a$, $b$, $c\in V$)\\
(1) If $n>3$,  $\alpha_n(a\circ b, c)=0$ ;\\
(2) If $n\leq 3$,
\begin{gather}
\label{eqqq1}\sum_{i=0}^3\lambda^i\alpha_i(a,b)=-\sum_{i=0}^3(-\lambda-\beta)^{i}\alpha_i(b,a), \\
\label{eqqe1}\alpha_3(a,c\circ b)=\alpha_3(a\circ b,c)=\alpha_3(b\circ a, c),\\
\label{eqqx2}\alpha_2(a,c\circ b)+\beta \alpha_3(a, c\circ b)+\alpha_3(a,[c,b])=\alpha_2(a\circ b,c)+\alpha_3([b,a],c),\\
\label{eqq3}\alpha_2(a,b\ast c)+\alpha_2(b\circ a,c)=2\alpha_2(a\circ b,c)+3\alpha_3([b,a],c),\\
\label{eqq4}\alpha_1(a,c\circ b)+\beta \alpha_2(a,c\circ b)+\alpha_2(a,[c,b])=\alpha_1(a\circ b,c)+\alpha_2([b,a],c),\\
\label{eqq5}\alpha_1(a,b\ast c)-\alpha_1(b,a\ast c)=-\alpha_1(b\circ a,c)+\alpha_1(a\circ b,c)+2\alpha_2([b,a],c),\\
\label{eqq6}\alpha_0(a,c\circ b)+\beta \alpha_1(a, c\circ b)+\alpha_1(a,[c,b])-\alpha_0(b, a\ast c)\\
=\alpha_0(a\circ b,c)+\alpha_1([b,a],c),\nonumber\\
\label{eqq7}\beta\alpha_0(a,c\circ b)-\beta\alpha_0(b,c\circ a)+\alpha_0(a,[c,b])-\alpha_0(b,[c,a])=\alpha_0([b,a],c);
\end{gather}
(3) Such two 2-cocycles $\alpha_\lambda(\cdot,\cdot)$ and $\alpha_\lambda^{'}(\cdot,\cdot)$ are equivalent if and only if there exists a linear map
$\varphi: V\rightarrow \mathbb{C}$ such that
\begin{eqnarray}
\alpha_\lambda(a,b)=\alpha_\lambda^{'}(a,b)+\beta \varphi(b\circ a)+\lambda \varphi(a\ast b)+\varphi([b,a]), ~~~~\text{for all $a$, $b\in V$}.
\end{eqnarray}
\end{theorem}
\begin{proof}
According to (\ref{e4}), we get $\sum_{i=0}^n\lambda^i\alpha_i(a,b)=-\sum_{i=0}^n(-\lambda-\beta)^i\alpha_i(b,a)$.

By (\ref{e3}) and (\ref{e4}), (\ref{ec1}) becomes
\begin{gather}
\label{cee}(\lambda+\beta)\alpha_\lambda(a,c\circ b)+\mu\alpha_\lambda(a, b\ast c)+\alpha_\lambda(a, [c,b])
-(\mu+\beta)\alpha_\mu(b,c\circ a)\\
-\lambda\alpha_\mu(b,a\ast c)-\alpha_\mu(b,[c,a])\nonumber\\
=(-\lambda-\mu)\alpha_{\lambda+\mu}(b\circ a, c)+\lambda \alpha_{\lambda+\mu}(a\ast b,c)+\alpha_{\lambda+\mu}([b,a],c).\nonumber
\end{gather}

If $n> 3$, by the assumption that $\alpha_\lambda(a,b)=\sum_{i=0}^n\lambda^i\alpha_i(a,b)$ and comparing the coefficients of $\lambda^2\mu^{n-1}$ and $\lambda^{n-1}\mu^2$, we obtain
\begin{eqnarray*}
n\alpha_n(a\circ b,c)-C_n^2\alpha_n(b\circ a,c)=0,~~C_n^2\alpha_n(a\circ b,c)-n\alpha_n(b\circ a,c)=0.
\end{eqnarray*}
Since $n>3$,  $\alpha_n(a\circ b, c)=0$ for any $a$, $b$, $c\in V$.

If $n\leq 3$, taking $\alpha_\lambda(a,b)=\sum_{i=0}^3\lambda^i\alpha_i(a,b)$ into (\ref{cee}) and comparing the coefficients of
$\lambda^4$, $\lambda\mu^3$, $\lambda^2\mu^2$, $\lambda^3$, $\lambda^2\mu$, $\lambda^2$, $\lambda\mu$, $\lambda$ and
$\lambda^0\mu^0$, we get
\begin{gather}
\label{eqe1}\alpha_3(a,c\circ b)=\alpha_3(a\circ b,c),\\
\label{eqe2}-\alpha_3(b,a\ast c)=-3\alpha_3(b\circ a,c)+\alpha_3(a\circ b,c),\\
\label{eqe3}\alpha_3(b\circ a,c)=\alpha_3(a\circ b,c),
\end{gather}
and (\ref{eqqe1})-(\ref{eqq7}).

By (\ref{eqe1}) and (\ref{eqe3}),
(\ref{eqe2}) is reduced to
\begin{eqnarray}
\label{eq16}\alpha_3(b,a\circ c)=\alpha_3(b\circ a,c).\end{eqnarray}
Since $\alpha_3(a,b)=\alpha_3(b,a)$, we can obtain (\ref{eq16}) from (\ref{eqqe1}) as follows:
$$\alpha_3(b,a\circ c)=\alpha_3(b\circ c,a)=\alpha_3(c\circ b,a)=\alpha_3(c,a\circ b)=\alpha_3(a\circ b,c)=\alpha_3(b\circ a,c).$$
Therefore, (\ref{eqe1})-(\ref{eqe3}) are equivalent to (\ref{eqqe1}). Then it is easy to see that
$\alpha_\lambda(\cdot,\cdot)=\sum_{i=0}^3\lambda^i\alpha_i(\cdot,\cdot)$ is a 2-cocycle if and only if
(\ref{eqqq1})-(\ref{eqq7}) hold.

Finally, by the discussion above this theorem, such two 2-cocycles $\alpha_\lambda(\cdot,\cdot)$ and $\alpha_\lambda^{'}(\cdot,\cdot)$ are equivalent if and only if there exists a $\mathbb{C}[\partial]$-module homomorphism $\varphi: R\rightarrow \mathbb{C}_\beta$ such that
$\alpha_\lambda(a,b)=\alpha_\lambda^{'}(a,b)+\varphi(\partial(b\circ a)+\lambda(a\ast b)+[b,a]), $ for all $a$, $b\in V$.
Since $\varphi$ is a $\mathbb{C}[\partial]$-module homomorphism, $\varphi(\partial(b\circ a))=\beta\varphi(b\circ a)$. Moreover, since $R=\mathbb{C}[\partial]V$ is free as a $\mathbb{C}[\partial]$-module, a $\mathbb{C}[\partial]$-module homomorphism $\varphi: R\rightarrow \mathbb{C}_\beta$ can be determined by the restricted linear map $\varphi\mid_V: V\rightarrow \mathbb{C}$. Then (3) can be directly obtained.
\end{proof}

\begin{remark}
By (3) in Theorem \ref{t1}, if there exists some nonzero $\alpha_i(\cdot,\cdot)$ for $i\geq 2$, then there are non-trivial central extensions of this Lie conformal algebra up to equivalence.
\end{remark}

\begin{remark}\label{rem1}
If we consider central extensions of quadratic Lie conformal algebras by an abelian Lie conformal algebra $\mathbb{C}[\partial]\mathfrak{c}$ which is free of rank one as a $\mathbb{C}[\partial]$-module, Theorem \ref{t1} also holds, when we replace $\mathfrak{c}_\beta$ by $\mathbb{C}[\partial]\mathfrak{c}$, $\beta$ by $\partial$, $\alpha_\lambda(a,b)\in \mathbb{C}[\lambda]$ by $\alpha_\lambda(a,b)\in \mathbb{C}[\lambda,\partial]$ for any $a$, $b\in V$,
$\mathbb{C}$ in (3) by $\mathbb{C}[\partial]$. Note that in this case, $\alpha_i(a,b)\in \mathbb{C}[\partial]$ for any $a$, $b\in V$ and $i\in \{0,\cdots, n\}$.
\end{remark}
\begin{corollary}\label{co1}
Let $R=\mathbb{C}[\partial]V$ be a finite quadratic Lie conformal algebra corresponding to the Gel'fand-Dorfman bialgebra $(V, \circ, [\cdot,\cdot])$ with $V=V\circ V$. Set $\widehat{R}=R\oplus\mathbb{C}\mathfrak{c}_\beta$ be a central extension of $(R,[\cdot_\lambda \cdot])$
with the $\lambda$-bracket given by (\ref{e2}). Then we obtain $\alpha_\lambda(a,b)=\sum_{i=0}^3\lambda^i\alpha_i(a,b)$ for any $a$, $b\in V$, where $\alpha_i(\cdot,\cdot):V\times V\rightarrow \mathbb{C}$ are bilinear forms on $V$ for any $i\in\{0,1,2,3\}$ and they satisfy (\ref{eqqq1})-(\ref{eqq7}).  Moreover,  such two 2-cocycles $\alpha_\lambda(\cdot,\cdot)$ and $\alpha_\lambda^{'}(\cdot,\cdot)$ are equivalent if and only if $\alpha_i(a,b)=\alpha_i^{'}(a,b)$ for $i=2$ and $i=3$, and there exists a linear map
$\varphi: V\rightarrow \mathbb{C}$ such that
\begin{eqnarray}
\alpha_1(a,b)=\alpha_1^{'}(a,b)+\varphi(a\ast b),\\
\alpha_0(a,b)=\alpha_0^{'}(a,b)+\beta\varphi(b\circ a)+\varphi([b,a]).
\end{eqnarray}
for all $a$, $b\in V$.
\end{corollary}
\begin{proof}
Since $R=\mathbb{C}[\partial]V$ is a finitely generated and free $\mathbb{C}[\partial]$-module, we can assume that $\alpha_\lambda(a,b)=\sum_{i=0}^n\lambda^i\alpha_i(a,b)$ for any $a$, $b\in V$ and some non-negative integer $n$.
Since $V=V\circ V$, for any $x\in V$, there exist some $m$ and $y_i$, $z_i\in V$ such that $x=\sum_{i=0}^my_i\circ z_i$.
When $n>3$, by Theorem \ref{t1}, we can obtain that for any $x$, $a\in V$, $\alpha_n(x, a)=\sum_{i=0}^m\alpha_n(y_i\circ z_i, a)=0$. Therefore, if $n>3$, $\alpha_n(a,b)=0$ for any $a$, $b\in V$.
So, by Theorem \ref{t1}, $n\leq 3$ and (\ref{eqqq1})-(\ref{eqq7}) hold for any
$a$, $b$, $c\in V$. Then we obtain this corollary.
\end{proof}

\begin{corollary}\label{c1}
Let $R=\mathbb{C}[\partial]V$ be a quadratic Lie conformal algebra corresponding to the Novikov algebra $(V, \circ)$ with $V=V\circ V$. Set $\widehat{R}=R\oplus\mathbb{C}\mathfrak{c}_\beta$ be a central extension of $(R,[\cdot_\lambda \cdot])$ with the $\lambda$-bracket of $\widehat{R}$ given by
\begin{eqnarray}
\label{ex2}\widetilde{[a_\lambda b]}=\partial(b\circ a)+\lambda(a\ast b)+\alpha_\lambda(a,b)\mathfrak{c}_\beta,
\end{eqnarray}
where $a$, $b\in V$ and $\alpha_\lambda(a,b)\in \mathbb{C}[\lambda]$.
 Then we get $\alpha_\lambda(a,b)=\sum_{i=0}^3\lambda^i\alpha_i(a,b)$ for any $a$, $b\in V$, where $\alpha_i(\cdot,\cdot):V\times V\rightarrow \mathbb{C}$ are bilinear forms for any $i\in\{0,1,2,3\}$ and they satisfy (\ref{eqqq1}), (\ref{eqqe1}) and
\begin{eqnarray}
\label{exx1}\alpha_2(a,c\circ b)+\beta \alpha_3(a, c\circ b)=\alpha_2(a\circ b,c),\\
\label{exx2}\alpha_2(a,b\ast c)+\alpha_2(b\circ a,c)=2\alpha_2(a\circ b,c),\\
\label{exx3}\alpha_1(a,c\circ b)+\beta \alpha_2(a,c\circ b)=\alpha_1(a\circ b,c),\\
\label{exx4}\alpha_1(a,b\ast c)-\alpha_1(b,a\ast c)=-\alpha_1(b\circ a,c)+\alpha_1(a\circ b,c),\\
\label{exx5}\alpha_0(a,c\circ b)+\beta \alpha_1(a, c\circ b)-\alpha_0(b, a\ast c)
=\alpha_0(a\circ b,c),\\
\label{exx6}\beta(\alpha_0(a,c\circ b)-\alpha_0(b,c\circ a))=0,
\end{eqnarray}
for any
$a$, $b$, $c\in V$. In particular, when $\beta=0$, (\ref{exx1})-(\ref{exx6}) are equivalent to
\begin{eqnarray}
\label{eq1}&&\alpha_2(a,b\circ c)+\alpha_2(b\circ a,c)=\alpha_2(a\circ b, c)=\alpha_2(a,c\circ b),\\
\label{eq2}&&\alpha_1(a,c\circ b)=\alpha_1(a\circ b,c),\\
\label{eq3}&&\alpha_0(c\circ a,b)-\alpha_0(c\circ b,a)=\alpha_0(a\circ b,c)-\alpha_0(a\circ c,b)
\end{eqnarray}
Moreover,  such two 2-cocycles $\alpha_\lambda(\cdot,\cdot)$ and $\alpha_\lambda^{'}(\cdot,\cdot)$ are equivalent if and only if $\alpha_i(a,b)=\alpha_i^{'}(a,b)$ for $i=2$ and $i=3$, and there exists a linear map
$\varphi: V\rightarrow \mathbb{C}$ such that
\begin{eqnarray}
\alpha_1(a,b)=\alpha_1^{'}(a,b)+\varphi(a\ast b),\\
\alpha_0(a,b)=\alpha_0^{'}(a,b)+\beta\varphi(b\circ a),
\end{eqnarray}
for all $a$, $b\in V$.
\end{corollary}
\begin{proof}
For any $a$, $b\in V$, set $\alpha_\lambda(a,b)=\sum_{i=0}^{n_{a,b}}\lambda^i\alpha_i(a,b)$ where
$\alpha_i(\cdot,\cdot)$ are bilinear forms on $V$ and $n_{a,b}$ is a non-negative integer depending on $a$ and $b$.
With the same process as in Theorem \ref{t1}, (\ref{cee}) becomes
\begin{gather}
(\lambda+\beta)\alpha_\lambda(a,c\circ b)+\mu\alpha_\lambda(a, b\ast c)
-(\mu+\beta)\alpha_\mu(b,c\circ a)-\lambda\alpha_\mu(b,a\ast c)\nonumber\\
\label{tt1}=(-\lambda-\mu)\alpha_{\lambda+\mu}(b\circ a, c)+\lambda \alpha_{\lambda+\mu}(a\ast b,c).
\end{gather}
For fixed $a$, $b$, $c$, there are only finite elements of $V$ appearing in $\alpha_\lambda(\cdot,\cdot)$ in (\ref{tt1}). Therefore, we may assume the degrees of all $\alpha_\lambda(\cdot,\cdot)$ in (\ref{tt1}) are smaller than some non-negative integer. So, we set $\alpha_\lambda(a,c\circ b)
=\sum_{i=0}^n\lambda^i\alpha_i(a, c\circ b)$, $\cdots$ and
$\alpha_{\lambda+\mu}(a\ast b,c)=\sum_{i=0}^n(\lambda+\mu)^i\alpha_i(a\ast b,c)$. Of course, here, $n$ depends on
$a$, $b$ and $c$.

If $n>3$, by comparing the coefficients of $\lambda^2\mu^{n-1}$ and $\lambda^{n-1}\mu^2$ in (\ref{tt1}), we get
\begin{eqnarray*}
n\alpha_n(a\circ b,c)-C_n^2\alpha_n(b\circ a,c)=0,~~C_n^2\alpha_n(a\circ b,c)-n\alpha_n(b\circ a,c)=0.
\end{eqnarray*}
Therefore, $\alpha_n(a\circ b,c)=\alpha_n(b\circ a,c)=0$. Repeating this process, we can get
$\alpha_m(a\circ b,c)=\alpha_m(b\circ a,c)=0$ for all $n\geq m> 3$.

By the discussion above, for any $a$, $b$, $c\in V$, we get $\alpha_m(a\circ b,c)=0$ for all $m>3$.
Since for any $x\in V$, there exist some $m$ and $y_i$, $z_i\in V$ such that $x=\sum_{i=0}^my_i\circ z_i$, we get
$\alpha_m(x,c)=0$ for all $m>3$ and any $c\in V$. Hence, $\alpha_\lambda(a,b)=\sum_{i=0}^3 \lambda^i\alpha_i(a,b)$. Moreover,
it is easy to see that (\ref{exx1})-(\ref{exx6}) are equivalent to (\ref{eqqx2})-(\ref{eqq7}) with
$[\cdot,\cdot]$ trivial, and (\ref{eq1})-(\ref{eq3}) are equivalent to (\ref{exx1})-(\ref{exx6}) when $\beta=0$.

By now, the proof is finished.
\end{proof}

\begin{remark}
Note that Corollary \ref{c1} also holds when $V$ is infinite-dimensional.
\end{remark}

\begin{corollary}\label{co2}
If for Novikov algebra $(V,\circ)$, there is a right unit $e\in V$, then the central extensions of quadratic Lie conformal algebra $R=\mathbb{C}[\partial]V$ by $\mathbb{C}\mathfrak{c}_\beta$ with $\beta\neq 0$ are trivial, i.e.
$H^2(R,\mathbb{C}\mathfrak{c}_\beta)=0$.

When $\beta=0$ and $(V, \circ)$ is a Novikov algebra with a unit $1$ in Corollary \ref{c1}, then $\alpha_2(a, b)=\alpha_0(a,b)=0$ for any $a$, $b\in V$.
\end{corollary}
\begin{proof}
Obviously, if $(V, \circ)$ has a right unit $e$ or a unit $1$, $V=V\circ V$.

When $\beta\neq 0$, setting $b=e$ in (\ref{exx1}), (\ref{exx3}),
(\ref{exx5}) and (\ref{exx6}), we get that $\alpha_3(a,c)=\alpha_2(a,c)=0$, $\alpha_1(a,c)=\frac{\alpha_0(e,c\ast a)}{\beta}$, and $\alpha_0(a,c)=\alpha_0(e,c\circ a)$ for any $a$, $c\in V$. Set
$\varphi(a)=\frac{\alpha_0(e,a)}{\beta}$ for any $a\in V$ in Corollary \ref{c1}. By Corollary \ref{c1}, we can make $\alpha_1(\cdot,\cdot)$ and $\alpha_0(\cdot,\cdot)$ be zero. Therefore, in this case, $H^2(R,\mathbb{C}\mathfrak{c}_\beta)=0$.

When $\beta=0$ and $(V, \circ)$ is a Novikov algebra with a unit $1$, letting $b=1$ in (\ref{eq1}), we can  directly obtain
$\alpha_2(a,c)=0$ for any $a$, $c\in V$. Similarly, by (\ref{eqqq1}), we get $\alpha_0(a,b)=-\alpha_0(b,a)$ for any $a$, $b\in V$. Letting $c=b=1$ in (\ref{eq3}) and
using $\alpha_0(a,1)=-\alpha_0(1,a)$, we have $\alpha_0(a,1)=0$ for any $a\in V$. Then setting $c=1$ in (\ref{eq3}) and using $\alpha_0(a,1)=0$ for any $a\in V$, $\alpha_0(a,b)=0$ is obtained for any $a$, $b\in V$.
\end{proof}
\begin{remark}
By Remark \ref{rem1}, Corollary \ref{c1} also holds if we replace $\mathfrak{c}_\beta$ by $\mathbb{C}[\partial]\mathfrak{c}$, $\beta$ by $\partial$ and $\alpha_\lambda(a,b)\in \mathbb{C}[\lambda]$ by $\alpha_\lambda(a,b)\in \mathbb{C}[\lambda,\partial]$ for any $a$, $b\in V$.

According to Corollary \ref{co2}, if Novikov algebra $(V,\circ)$ has a right unit, then the central extensions of quadratic Lie conformal algebra $R=\mathbb{C}[\partial]V$ by an abelian Lie conformal algebra $\mathbb{C}[\partial]\mathfrak{c}$ are trivial, i.e. $H^2(R, \mathbb{C}[\partial]\mathfrak{c})=0$.
\end{remark}
\begin{proposition}\label{pro1}
Let $(V, \circ, [\cdot,\cdot])$ be a Gel'fand-Dorfman bialgebra and $\alpha_i(\cdot,\cdot)$ $(i=0, 1, 2, 3)$ be bilinear forms
on $V$ satisfying (\ref{eqqq1})-(\ref{eqq7}) for any $a$, $b$, $c\in V$ with $\beta=0$. In addition, let $\pi: \mathcal{L}(V)\times \mathcal{L}(V)\rightarrow \mathbb{C}$ be the bilinear form on $\mathcal{L}(V)$ defined by
\begin{eqnarray*}
\pi(a_m, b_n)&=&\alpha_0(a,b)\delta_{m+n+1,0}+m\alpha_1(a,b)\delta_{m+n,0}\\
&&+m(m-1)\alpha_2(a,b)\delta_{m+n-1,0}+m(m-1)(m-2)\alpha_3(a,b)\delta_{m+n-2,0},
\end{eqnarray*}
for $a$, $b\in V$, $m$, $n\in \mathbb{Z}$. Then $\pi$ is a 2-cocycle of Lie algebra $\mathcal{L}(V)$.
\end{proposition}
\begin{proof}
Let $R=\mathbb{C}[\partial]V$ be a quadratic Lie conformal algebra corresponding to the Gel'fand-Dorfman bialgebra $(V, \circ, [\cdot,\cdot])$. Denote $\mathfrak{c}_0$ by $\mathfrak{c}$.
Assume that $\widehat{R}$ is the central extension defined by (\ref{e2}). It is easy to see that the coefficient algebra of $\widehat{R}$ is $\text{Coeff}(\widehat{R})=\mathcal{L}(V)\oplus \mathbb{C}\mathfrak{c}_{-1}$ with the following Lie bracket: $[a_m,{\mathfrak{c}}_{-1}]=0$,  and
\begin{gather}
[a_m, b_n]=[a,b]_{m+n}+m(a\circ b)_{m+n-1}-n(b\circ a)_{m+n-1}+(\alpha_0(a,b)\delta_{m+n+1,0}\nonumber\\
+m\alpha_1(a,b)\delta_{m+n,0}+m(m-1)\alpha_2(a,b)\delta_{m+n-1,0}\nonumber\\
\label{gt1}+m(m-1)(m-2)\alpha_3(a,b)\delta_{m+n-2,0}){\mathfrak{c}}_{-1},
\end{gather}
for any $a$, $b\in V$, and $m$, $n\in \mathbb{Z}$. Obviously, $\text{Coeff}(\widehat{R})$ is the central extension of $\mathcal{L}(V)$ by a one-dimensional center $\mathbb{C}\mathfrak{c}_{-1}$.
Therefore, by $(\ref{gt1})$, $\pi$ is a 2-cocycle of $\mathcal{L}(V)$.
\end{proof}

\begin{corollary}
Let $(V, \circ)$ be a Novikov algebra and $\alpha_i(\cdot,\cdot)$ $(i=0, 1, 2, 3)$ be bilinear forms
on $V$ satisfying  (\ref{eqqq1}), (\ref{eqqe1}) and (\ref{eq1})-(\ref{eq3}) for any $a$, $b$, $c\in V$ with $\beta=0$. Moreover, let $\pi_i: \mathcal{L}(V)\times \mathcal{L}(V)\rightarrow \mathbb{C}$ be bilinear forms on $\mathcal{L}(V)$ defined by
\begin{eqnarray}
&&\pi_0(a\otimes t^m, b\otimes t^n)=\alpha_0(a,b)\delta_{m+n+1,0},\\
&&\pi_1(a\otimes t^m, b\otimes t^n)=m\alpha_1(a,b)\delta_{m+n,0},\\
&&\pi_2(a\otimes t^m, b\otimes t^n)=m(m-1)\alpha_2(a,b)\delta_{m+n-1,0},\\
&&\pi_3(a\otimes t^m, b\otimes t^n)=m(m-1)(m-2)\alpha_3(a,b)\delta_{m+n-2,0},
\end{eqnarray}
for $a$, $b\in V$, $m$, $n\in \mathbb{Z}$. Then $\pi_0$, $\pi_1$, $\pi_2$, $\pi_3$ are 2-cocycles of Lie algebra $\mathcal{L}(V)$.
\end{corollary}
\begin{proof}
According to that $\alpha_0(\cdot,\cdot)$, $\alpha_1(\cdot,\cdot)$, $\alpha_2(\cdot,\cdot)$ and $\alpha_3(\cdot,\cdot)$ donot depend on each other, this corollary can be directly obtained from Proposition \ref{pro1}.
\end{proof}
\begin{remark}
This corollary can also be referred to Proposition 2.5 in \cite{PB3}.
\end{remark}

By Corollary \ref{co2}, when $\beta\neq 0$, $H^2(R,\mathbb{C}\mathfrak{c}_\beta)=0$ for any quadratic Lie conformal corresponding to Novikov algebra $(V,\circ)$ with a right unit. In the following, we mainly focus on the case when $\beta=0$, i.e. we use Theorem \ref{t1} to determine central extensions of several specific Lie conformal algebras by a one-dimensional center $\mathbb{C}\mathfrak{c}$ up to equivalence.

\begin{example}\label{ex1}
Let $R(\alpha,\beta)=\mathbb{C}[\partial]L\oplus\mathbb{C}[\partial]W$ be a Lie conformal algebra with the $\lambda$-bracket given as follows:
\begin{eqnarray}
[L_\lambda L]=(\partial+2\lambda)L,~~[L_\lambda W]=(\partial+\alpha\lambda+\beta)W,~~[W_\lambda W]=0,
\end{eqnarray}
for some $\alpha$, $\beta\in \mathbb{C}$. In fact, it is the semi-direct sum of $\text{Vir}$ and the conformal module $\mathbb{C}[\partial]W$ of $\text{Vir}$ with the action $L_\lambda(W)=(\partial+\alpha\lambda+\beta)W$. Moreover, it also appeared in \cite{HF}.

Next, we begin to study
central extensions of $R(\alpha,\beta)$ by a one-dimensional center $\mathbb{C}\mathfrak{c}$ up to equivalence.

Obviously, $R(\alpha,\beta)$ is a quadratic Lie conformal algebra corresponding to the 2-dimensional Gel'fand-Dorfman bialgebra
$V=\mathbb{C}L\oplus\mathbb{C}W$ with the Novikov operation ``$\circ$" and Lie bracket defined as follows:
\begin{eqnarray}
&&L\circ L=L,~~~L\circ W=(\alpha-1)W,~~~W\circ L=W,~~~W\circ W=0,\\
&&[L,L]=0,~~~[W,L]=\beta W,~~~[W,W]=0.
\end{eqnarray}

It is easy to see that $(V, \circ, [\cdot,\cdot])$ satisfies the assumption in Corollary \ref{co1}. Therefore, according to Corollary \ref{co1}, (\ref{eqqq1})-(\ref{eqq7}) and by some simple computations, we can obtain\\
(1) If $\alpha\neq2$, $\alpha\neq 1$, $\alpha\neq 0$, or $\alpha=1$, $\beta\neq 0$ or $\alpha=2$, $\beta\neq 0$,
\begin{eqnarray*}
&&\alpha_3(L,L)=A,~~\alpha_1(L,L)=B,~~\alpha_1(L,W)=C,~~\alpha_0(L,W)=\frac{\beta}{\alpha}C,\\
&&\alpha_3(L,W)=\alpha_3(W,W)=\alpha_2(L,W)=\alpha_2(L,L)=\alpha_2(W,W)=0,\\
&&\alpha_1(W,W)=\alpha_0(W,W)=\alpha_0(L,L)=0,
\end{eqnarray*}
for any $A$, $B$, $C\in \mathbb{C}$;\\
(2) If  $\alpha=0$ and $\beta=0$,
\begin{eqnarray*}
&&\alpha_3(L,L)=A,~~\alpha_1(L,L)=B,~~\alpha_1(L,W)=C,~~\alpha_0(L,W)=D,\\
&&\alpha_3(L,W)=\alpha_3(W,W)=\alpha_2(L,W)=\alpha_2(L,L)=\alpha_2(W,W)=0,\\
&&\alpha_1(W,W)=\alpha_0(W,W)=\alpha_0(L,L)=0,
\end{eqnarray*}
for any $A$, $B$, $C$, $D\in \mathbb{C}$;\\
(3) If  $\alpha=0$ and $\beta\neq 0$,
\begin{eqnarray*}
&&\alpha_3(L,L)=A,~~\alpha_1(L,L)=B,~~\alpha_0(L,W)=C,\\
&&\alpha_3(L,W)=\alpha_3(W,W)=\alpha_2(L,W)=\alpha_2(L,L)=\alpha_2(W,W)=0,\\
&&\alpha_1(L,W)=\alpha_1(W,W)=\alpha_0(W,W)=\alpha_0(L,L)=0,
\end{eqnarray*}
for any $A$, $B$, $C\in \mathbb{C}$;\\
(4) If  $\alpha=1$, $\beta=0$,
\begin{eqnarray*}
&&\alpha_3(L,L)=A,~~\alpha_2(L,W)=B,~~\alpha_1(L,L)=C,~~\alpha_1(L,W)=D,~~\alpha_1(W,W)=E,\\
&&\alpha_3(L,W)=\alpha_3(W,W)=\alpha_2(W,W)=\alpha_2(L,L)=0,\\
&&\alpha_0(W,W)=\alpha_0(L,L)=\alpha_0(L,W)=0,
\end{eqnarray*}
for any $A$, $B$, $C$, $D$, $E\in \mathbb{C}$;\\
(5) If  $\alpha=2$, $\beta=0$,
\begin{eqnarray*}
&&\alpha_3(L,L)=A,~~\alpha_3(L,W)=B,~~\alpha_1(L,L)=C,~~\alpha_1(L,W)=D,~~\alpha_1(W,W)=0,\\
&&\alpha_3(W,W)=\alpha_2(L,W)=\alpha_2(W,W)=\alpha_2(L,L)=0,\\
&&\alpha_0(W,W)=\alpha_0(L,W)=\alpha_0(L,L)=0,
\end{eqnarray*}
for any $A$, $B$, $C$, $D\in \mathbb{C}$.

Therefore, if $\alpha\neq2$, $\alpha\neq 1$, $\alpha\neq 0$, or $\alpha=1$, $\beta\neq 0$ or $\alpha=2$, $\beta\neq 0$,
by choosing the linear map $\varphi:V\rightarrow\mathbb{C}$ in Corollary \ref{co1} defined by
$\varphi(L)=\frac{B}{2}$ and $\varphi(W)=\frac{C}{\alpha}$, we can make $\alpha_1(\cdot,\cdot)$ and $\alpha_0(\cdot,\cdot)$ be zero up to equivalence. Therefore, by Corollary \ref{co1}, all equivalence classes of central extensions of
$R(\alpha,\beta)$ by a one-dimensional center $\mathbb{C}\mathfrak{c}$ are $\widetilde{R(\alpha,\beta)}(A)$  with the $\lambda$-brackets as follows:
\begin{eqnarray*}
&&[L_\lambda L]=(\partial+2\lambda)L+A\lambda^3 \mathfrak{c},\\
&&[L_\lambda W]=(\partial+\alpha\lambda+\beta)W,~~[W_\lambda W]=0,
\end{eqnarray*}
for all $A\in \mathbb{C}$. Note that if $A_1\neq A_2$, then $\widetilde{R(\alpha,\beta)}(A_1)$ is not equivalent to $\widetilde{R(\alpha,\beta)}(A_2)$. Therefore, in this case, $\text{dim}~H^2(R(\alpha,\beta),\mathbb{C}\mathfrak{c})=1$.
\\
If $\alpha=0$, and $\beta=0$, by choosing the linear map $\varphi:V\rightarrow\mathbb{C}$ in Corollary \ref{co1} defined by
$\varphi(L)=\frac{B}{2}$ and $\varphi(W)=0$, we can make $\alpha_1(L,L)$ be zero up to equivalence. Therefore, by Corollary \ref{co1}, all equivalence classes of central extensions of
$R(0,0)$ by a one-dimensional center $\mathbb{C}\mathfrak{c}$ are $\widetilde{R(0,0)}(A,C,D)$  with the $\lambda$-brackets as follows:
\begin{eqnarray*}
&&[L_\lambda L]=(\partial+2\lambda)L+A\lambda^3\mathfrak{c},\\
&&[L_\lambda W]=\partial W+(C\lambda+D)\mathfrak{c},~~[W_\lambda W]=0,
\end{eqnarray*}
for all $A$, $C$, $D\in \mathbb{C}$. Note that if $(A_1, C_1, D_1)\neq (A_2, C_2, D_2)$, then $\widetilde{R(0,0)}(A_1, C_1, D_1)$ is not equivalent to $\widetilde{R(0,0)}(A_2, C_2, D_2)$. Therefore, $\text{dim} ~H^2(R(0,0),\mathbb{C}\mathfrak{c})=3$.\\
If $\alpha=0$, and $\beta\neq0$, by choosing the linear map $\varphi:V\rightarrow\mathbb{C}$ in Corollary \ref{co1} defined by
$\varphi(L)=\frac{B}{2}$ and $\varphi(W)=\frac{C}{\beta}$, we can make $\alpha_1(\cdot,\cdot)$ and $\alpha_0(\cdot,\cdot)$ be zero up to equivalence.
Therefore, by Corollary \ref{co1}, all equivalence classes of central extensions of
$R(0,\beta)$ by a one-dimensional center $\mathbb{C}\mathfrak{c}$ are $\widetilde{R(0,\beta)}(A)$  with the $\lambda$-brackets as follows:
\begin{eqnarray*}
&&[L_\lambda L]=(\partial+2\lambda)L+A\lambda^3\mathfrak{c},\\
&&[L_\lambda W]=(\partial+\beta)W,~~[W_\lambda W]=0,
\end{eqnarray*}
for all $A\in \mathbb{C}$. Note that if $A_1\neq A_2$, then $\widetilde{R(0,\beta)}(A_1)$ is not equivalent to $\widetilde{R(0,\beta)}(A_2)$. Therefore, in this case, $\text{dim}~H^2(R(0,\beta),\mathbb{C}\mathfrak{c})=1$.\\
If $\alpha=1$ and $\beta=0$, by choosing the linear map $\varphi: V\rightarrow\mathbb{C}$ in Corollary \ref{co1} defined by
$\varphi(L)=\frac{C}{2}$ and $\varphi(W)=D$, we can make $\alpha_1(L,L)$ and $\alpha_1(L,W)$ be zero up to equivalence. Therefore, by Corollary \ref{co1}, all equivalence classes of central extensions of
$R(1,0)$ by a one-dimensional center $\mathbb{C}\mathfrak{c}$ are $\widetilde{R(1,0)}(A,B,E)$  with the $\lambda$-brackets as follows:
\begin{eqnarray*}
&&[L_\lambda L]=(\partial+2\lambda)L+A\lambda^3\mathfrak{c},\\
&&[L_\lambda W]=(\partial+\lambda)W+B\lambda^2\mathfrak{c},~~[W_\lambda W]=E\lambda\mathfrak{c},
\end{eqnarray*}
for all $A$, $B$, $E\in \mathbb{C}$. Note that if $(A_1, B_1, E_1)\neq (A_2, B_2, E_2)$, then $\widetilde{R(1,0)}(A_1, B_1, E_1)$ is not equivalent to $\widetilde{R(1,0)}(A_2, B_2, E_2)$. Therefore, $\text{dim} ~H^2(R(1,0),\mathbb{C}\mathfrak{c})=3$.\\
If $\alpha=2$ and $\beta=0$, by choosing the linear map $\varphi:V\rightarrow\mathbb{C}$ in Corollary \ref{co1} defined by
$\varphi(L)=\frac{C}{2}$ and $\varphi(W)=\frac{D}{2}$, we can make $\alpha_1(L,L)$ and $\alpha_1(L,W)$ be zero up to equivalence. Therefore, by Corollary \ref{co1}, all equivalence classes of central extensions of
$R(2,0)$ by a one-dimensional center $\mathbb{C}\mathfrak{c}$ are $\widetilde{R(2,0)}(A,B)$  with the $\lambda$-brackets as follows:
\begin{eqnarray*}
&&[L_\lambda L]=(\partial+2\lambda)L+A\lambda^3\mathfrak{c},\\
&&[L_\lambda W]=(\partial+2\lambda)W+B\lambda^3\mathfrak{c},~~[W_\lambda W]=0,
\end{eqnarray*}
for all $A$, $B\in \mathbb{C}$. Note that if $(A_1, B_1)\neq (A_2, B_2)$, then $\widetilde{R(2,0)}(A_1, B_1)$ is not equivalent to $\widetilde{R(2,0)}(A_2, B_2)$. Therefore,  $\text{dim} ~H^2(R(2,0),\mathbb{C}\mathfrak{c})=2$.\\
\end{example}

\begin{example}\label{l1}
The loop Virasoro Lie conformal algebra $\mathcal{LW}$ introduced in \cite{WCY} is a Lie conformal algebra with the $\mathbb{C}[\partial]$-basis $\{L_i|i\in\mathbb{Z}\}$ and the $\lambda$-bracket given as follows:
\begin{eqnarray}
[{L_i}_\lambda L_j]=(-\partial-2\lambda)L_{i+j}.
\end{eqnarray}
Its coefficient algebra is just the loop Virasoro Lie algebra. In \cite{WCY},  conformal derivations, rank
one conformal modules and $\mathbb{Z}$-graded free intermediate series modules of $\mathcal{LW}$ were determined. But, central extensions of
$\mathcal{LW}$ were not studied. Next, we will give a characterization of central extensions of
$\mathcal{LW}$ by a one-dimensional center $\mathbb{C}\mathfrak{c}$ up to equivalence.

Obviously, $\mathcal{LW}$ is a quadratic Lie conformal algebra corresponding to the infinite-dimensional Novikov algebra $V=\bigoplus_i \mathbb{C}L_i$ with the operation ``$\circ$" as follows:
\begin{eqnarray}
L_i\circ L_j=-L_{i+j},~~~~\text{for any $i$, $j\in \mathbb{Z}$.}
\end{eqnarray}

According to Corollary \ref{c1}, (\ref{eqqq1}), (\ref{eqqe1}), (\ref{eq1})-(\ref{eq3}) and by some simple computations, we can obtain
\begin{eqnarray}
\label{eqq1}&&\alpha_3(L_i,L_{j+k})=\alpha_3(L_{i+k},L_{j}),~~~\alpha_3(L_i,L_j)=\alpha_3(L_j,L_i),\\
&&\alpha_2(L_i,L_j)=0,\\
\label{eqq2}&&\alpha_1(L_i,L_{j+k})=\alpha_1(L_{i+j},L_{k}),~~~\alpha_1(L_i,L_j)=\alpha_1(L_j,L_i),\\
\label{eqq3}&&2\alpha_0(L_{i+k},L_j)=\alpha_0(L_{i+j},L_k)+\alpha_0(L_{j+k},L_i),~~\alpha_0(L_i,L_j)=-\alpha_0(L_j,L_i),
\end{eqnarray}
for any $i$, $j$, $k\in \mathbb{Z}$.
By (\ref{eqq1}), we get $\alpha_3(L_i,L_{j+k})=\alpha_3(L_{i+j+k},L_{0})=\alpha_3(L_{0},L_{i+j+k})$. Therefore, we set
\begin{eqnarray}
\label{eqq4}\alpha_3(L_{0},L_{i+j+k})=f(i+j+k),
\end{eqnarray} for some complex function $f$. Thus, $\alpha_3(L_i,L_{j+k})=f(i+j+k)$. Letting $k=0$, we obtain
$\alpha_3(L_i,L_j)=f(i+j)$. Then, (\ref{eqq1}) holds for any $i$, $j$, $k\in \mathbb{Z}$.  Similarly, $\alpha_1(L_i,L_j)=g(i+j)$ for some complex function $g$.
Set $j=0$ in (\ref{eqq3}). We have $2\alpha_0(L_{i+k},L_0)=\alpha_0(L_{i},L_k)+\alpha_0(L_{k},L_i)=0$ for any $i$, $k\in \mathbb{Z}$.
Therefore, $\alpha_0(L_i,L_0)=\alpha_0(L_0,L_i)=0$. Then, set $k=0$ in (\ref{eqq3}). we get $$2\alpha_0(L_{i},L_j)=\alpha_0(L_{i+j},L_0)+\alpha_0(L_{j},L_i)=\alpha_0(L_{j},L_i)=-\alpha_0(L_{i},L_j).$$
Thus, $\alpha_0(L_{i},L_j)=0$.

Therefore, by choosing the linear map $\varphi: V\rightarrow\mathbb{C}$ in Corollary \ref{c1} defined by
$\varphi(L_i)=-\frac{g(i)}{2}$ for all $i\in \mathbb{Z}$, we can make $\alpha_1(\cdot,\cdot)$ be zero up to equivalence. Therefore, by Corollary \ref{c1}, all equivalence classes of central extensions of $\mathcal{LW}$ by a one-dimensional center $\mathbb{C}\mathfrak{c}$ are $\widetilde{\mathcal{LW}}(f)=\mathcal{LW}\oplus \mathbb{C}\mathfrak{c}$ with the $\lambda$-brackets as follows:
\begin{eqnarray}
&&[{L_i}_{\lambda}{L_j}]=(-\partial-2\lambda)L_{i+j}+f(i+j)\lambda^3\mathfrak{c},
\end{eqnarray}
for all complex functions $f$. Note that if $f_1\neq f_2$, then $\widetilde{\mathcal{LW}}(f_1)$ is not equivalent to $\widetilde{\mathcal{LW}}(f_2)$. Therefore, $\text{dim}~H^2(\mathcal{LW},\mathbb{C}\mathfrak{c})=\infty$.
\end{example}

\section{Conformal derivations of quadratic Lie conformal algebras}
In this section, we will investigate conformal derivations of quadratic Lie conformal algebras.

\begin{definition}
A \emph{conformal linear map} between $\mathbb{C}[\partial]$-modules $U$ and $V$ is a linear map
$\phi_{\lambda}: U\rightarrow V[\lambda]$ such that
\begin{eqnarray}
\phi_\lambda(\partial u)=(\partial+\lambda)\phi_\lambda u, ~~\text{for~~all~~$u\in U$.}
\end{eqnarray}
\end{definition}
We will abuse the notation by writing $\phi: U\rightarrow V$ any time it is clear from the context
that $\phi$ is conformal linear.

\begin{definition}
Let $R$ be a Lie conformal algebra. A conformal linear map $d: R\rightarrow R$ is
called a \emph{conformal derivation} of $R$ if
\begin{eqnarray}\label{d1}
d_\lambda[a_\mu b]=[(d_\lambda a)_{\lambda+\mu} b]+[a_\mu(d_\lambda b)], ~~\text{ $a$, $b\in R$.}
\end{eqnarray}
\end{definition}

The space of all conformal derivations of $R$ is denoted by $\text{CDer}(R)$. For any $a\in R$, there is a natural conformal derivation $\text{ad}~a: R\rightarrow R$ such that $$(\text{ad} ~a)_\lambda b=[a_\lambda b],~~~b\in R.$$
All conformal derivations of this kind are called \emph{inner}. The space of all inner conformal derivations is denoted by
$\text{CInn}(R)$.

Next, we present a characterization of conformal derivations of quadratic Lie conformal algebras.
\begin{theorem}\label{t2}
Let $R=\mathbb{C}[\partial]V$ be a quadratic Lie conformal algebra corresponding to the Gel'fand-Dorfman bialgebra $(V, \circ, [\cdot,\cdot])$ and $d$ be a conformal derivation of $R$.  Assume that \begin{eqnarray}\label{d2}
d_\lambda (a)=\sum_{i=0}^n\partial^id_\lambda^i(a),~~~\text{$d_\lambda^i(a)\in V[\lambda]$, $a\in V$, $i\in \{0,1,\cdots,n\}$},\end{eqnarray}
and there exists some $a\in V$ such that $d_\lambda^n(a)\neq 0$.
Then we obtain (for any $a$, $b\in V$)\\
(1) If $n> 3$, then $d_\lambda^n(b)\circ a=a\circ d_\lambda^n(b)=0$;\\
(2) If $n\leq 3$,
\begin{gather}
\label{so1}d_\lambda^3(b\circ a)=d_\lambda^3(a\circ b)=d_\lambda^3(b)\circ a,\\
\label{pw2}b\circ d_\lambda^3(a)=a\circ d_\lambda^3(b),\\
\label{pw3}2d_\lambda^3(b)\circ a+a\circ d_\lambda^3(b)=0,\\
\label{op1}\lambda d_\lambda^3(b\circ a)+d_\lambda^2(b\circ a)+d_\lambda^3[b,a]=[d_\lambda^3(b),a]+d_\lambda^2(b)\circ a,\\
\label{op2}d_\lambda^2(a\ast b)=2 d_\lambda^2(b)\circ a+d_\lambda^2(b)\ast a+3[d_\lambda^3(b),a],\\
\label{op3}-3\lambda b\circ d_\lambda^3(a)+b\circ d_\lambda^2(a)+d_\lambda^2(b)\circ a+2 d_\lambda^2(b)\ast a+3[d_\lambda^3(b),a]=0,\\
\label{op4}-4\lambda d_\lambda^3(a)\ast b+d_\lambda^2(a)\ast b-[b,d_\lambda^3(a)]+d_\lambda^2(b)\ast a
+[d_\lambda^3(b),a]=0,\\
\label{op5}\lambda d_\lambda^2(b\circ a)+d_\lambda^1(b\circ a)+d_\lambda^2[b,a]=d_\lambda^1(b)\circ a+[d_\lambda^2(b),a],\\
\label{op6}d_\lambda^1(a\ast b)=\sum_{i=1}^3(-1)^iC_{i}^1\lambda^{i-1}b\circ d_\lambda^i(a)+d_\lambda^1(b)\circ a+d_\lambda^1(b)\ast a
+2[d_\lambda^2(b),a],\\
\label{op7}\sum_{i=1}^3(-1)^iC_{i+1}^2\lambda^{i-1} d_\lambda^i(a)\ast b+\sum_{i=2}^3(-1)^iC_i^2\lambda^{i-2}[b,d_\lambda^i(a)]
+d_\lambda^1(b)\ast a+[d_\lambda^2(b),a]=0,\\
\label{op8}\lambda d_\lambda^1(b\circ a)+ d_\lambda^0(b\circ a)+d_\lambda^1[b,a]
=\sum_{i=0}^3(-1)^i\lambda^i b\circ d_\lambda^i(a)+d_\lambda^0(b)\circ a+[d_\lambda^1(b),a],\\
\label{op9}d_\lambda^0(a\ast b)=\sum_{i=0}^3(-1)^i C_{i+1}^1\lambda^i d_\lambda^i(a)\ast b+\sum_{i=1}^3(-1)^iC_i^1\lambda^{i-1}[b,d_\lambda^i(a)]+d_\lambda^0(b)\ast a
+[d_\lambda^1(b),a],\\
\label{so15}\lambda d_\lambda^0(b\circ a)+d_\lambda^0[b,a]
=\sum_{i=0}^3(-1)^i\lambda^{i+1}d_\lambda^i(a)\ast b+\sum_{i=0}^3(-1)^i\lambda^i[b,d_\lambda^i(a)]+[d_\lambda^0(b),a].
\end{gather}

\end{theorem}
\begin{proof}

Since $d$ is a conformal derivation of $R$, by (\ref{d1}) and (\ref{d2}),
we obtain
\begin{gather}
\label{g1}d_\lambda[a_\mu b]=d_\lambda(\partial(b\circ a)+\mu(a\ast b)+[b,a])\\
=(\lambda+\partial)d_\lambda(b\circ a)+\mu d_\lambda(a\ast b)+d_\lambda[b,a],\nonumber\\
=(\lambda+\partial)\sum_{i=0}^n\partial^id_\lambda^i(b\circ a)+\mu\sum_{i=0}^n\partial^id_\lambda^i(a\ast b)+\sum_{i=0}^n\partial^id_\lambda^i[b,a],\nonumber\end{gather}
and
\begin{gather}
\label{g2}[(d_\lambda a)_{\lambda+\mu}b]+[a_\mu(d_\lambda b)]=[(\sum_{i=0}^n\partial^id_\lambda^i(a))_{\lambda+\mu} b]+[a_\mu(\sum_{i=0}^n\partial^id_\lambda^i(b))]\\
=\sum_{i=0}^n(-\lambda-\mu)^i[d_\lambda^i(a)_{\lambda+\mu}b]+\sum_{i=0}^n(\mu+\partial)^i[a_\mu d_\lambda^i(b)]\nonumber\end{gather}
\begin{gather}=\sum_{i=0}^n(-\lambda-\mu)^i(\partial(b\circ d_\lambda^i(a))+(\lambda+\mu)d_\lambda^i(a)\ast b+[b,d_\lambda^i(a)])\nonumber\\
+\sum_{i=0}^n(\mu+\partial)^i(\partial(d_\lambda^i(b)\circ a)+\mu d_\lambda^i(b)\ast a+[d_\lambda^i(b),a]).\nonumber
\end{gather}

If $n> 3$, comparing the coefficients of $\mu^{n-1}\partial^2$, $\mu^2\partial^{n-1}$ in (\ref{g1}) and (\ref{g2}), we obtain the following equalities
\begin{eqnarray}
\label{rt1} nd_\lambda^n(b)\circ a+C_n^2d_\lambda^n(b)\ast a=0,~~~C_n^2d_\lambda^n(b)\circ a+nd_\lambda^n(b)\ast a=0.
\end{eqnarray}
Since $n>3$, by (\ref{rt1}), we get $d_\lambda^n(b)\circ a=d_\lambda^n(b)\ast a=0$ for any $a$, $b\in V$. Since $d_\lambda^n(b)\ast a=d_\lambda^n(b)\circ a+a\circ d_\lambda^n(b)$, we also have $a\circ d_\lambda^n(b)=0$.

If $n\leq 3$, assume
\begin{eqnarray}
\label{dd5}d_\lambda(a)=\sum_{i=0}^3\partial^i d_\lambda^i(a),~~\text{for~any~$a\in V$.}
\end{eqnarray} Taking (\ref{dd5}) into (\ref{g1}) and (\ref{g2}) and by comparing the coefficients of
$\partial^4$, $\partial^3\mu$, $\partial^2\mu^2$, $\partial\mu^3$, $\mu^4$, we can get
\begin{gather}
\label{soo1}d_\lambda^3(b\circ a)=d_\lambda^3(b)\circ a,\\
\label{so2}d_\lambda^3(a\ast b)=3d_\lambda^3(b)\circ a+d_\lambda^3(b)\ast a,\\
\label{so3}d_\lambda^3(b)\circ a+d_\lambda^3(b)\ast a=0,\\
\label{so4}-b\circ d_\lambda^3(a)+d_\lambda^3(b)\circ a+3 d_\lambda^3(b)\ast a=0,\\
\label{so5}-d_\lambda^3(a)\ast b+d_\lambda^3(b)\ast a=0.
\end{gather}
It is easy to check that (\ref{so1})-(\ref{pw3}) are equivalent to (\ref{soo1})-(\ref{so5}).
Similarly, by comparing the coefficients of $\partial^3$, $\mu\partial^2$,  $\mu^2\partial$, $\mu^3$, $\partial^2$, $\mu\partial$, $\mu^2$, $\partial$,$\mu$ and $\mu^0\partial^0$, we can immediately obtain
the equalities (\ref{op1})-(\ref{so15}).

By now, the proof is finished.
\end{proof}
\begin{remark}
By the definition of quadratic Lie conformal algebra, if there exists some nonzero $d^i$ in Theorem \ref{t2} for
$i\geq 2$, then there are non-inner conformal derivations of this Lie conformal algebra.
\end{remark}
\begin{corollary}\label{coo1}
Let $R=\mathbb{C}[\partial]V$ be a finite quadratic Lie conformal algebra corresponding to the Gel'fand-Dorfman bialgebra $(V,\circ,[\cdot,\cdot])$. Then we have\\
(1) If $(V,\circ)$ is a simple Novikov algebra, then any conformal derivation $d$ of
$R$ is of the form $d_\lambda(a)=\sum_{i=0}^3\partial^id_\lambda^i(a)$ for any $a\in V$, where $d_\lambda^i(a)\in V[\lambda]$ for $i\in \{0,1,2,3\}$ satisfy (\ref{so1})-(\ref{so15});\\
(2) If for $(V,\circ)$, there exists $x\in V$ such that $b\circ x=kb$ for any $b\in V$ and some fixed $k\in \mathbb{C}\backslash\{0\}$, then any conformal derivation $d$ of
$R$ is of the form $d_\lambda(a)=d_\lambda^0(a)+\partial d_\lambda^1(a)$ for any $a\in V$, where $d_\lambda^0(a)$, $d_\lambda^1(a)\in V[\lambda]$ and they satisfy \begin{gather}
\label{g3}d_\lambda^1(b\circ a)=d_\lambda^1(b)\circ a,\\
\label{g4}d_\lambda^1(a)\ast b=d_\lambda^1(b)\ast a,\\
\label{g5}d_\lambda^0(b\circ a)+\lambda d_\lambda^1(b\circ a)+d_\lambda^1[b,a]=b\circ d_\lambda^0(a)\\
-\lambda b\circ d_\lambda^1(a)+d_\lambda^0(b)\circ a
+[d_\lambda^1(b),a]\nonumber,\\
\label{g7}\lambda d_\lambda^0(b\circ a)+d_\lambda^0[b,a]=\lambda d_\lambda^0(a)\ast b-\lambda^2d_\lambda^1(a)\ast b\\+[b,d_\lambda^0(a)]-\lambda[b,d_\lambda^1(a)]+[d_\lambda^0(b),a]
\nonumber,\end{gather}
for any $a$, $b\in V$. Moreover, if $(V,\circ)$ also satisfies that there exists some $y\in V$, $y\circ V=V$, then any conformal derivation $d$ of
$R$ is of the form
$d=D+\widetilde{d^0}$ where $D$ is an inner conformal derivation and $\widetilde{d^0}$ satisfies
\begin{eqnarray}
\label{cr1}\widetilde{d}_\lambda^0(b\circ a)=b\circ \widetilde{d}_\lambda^0(a)+\widetilde{d}_\lambda^0(b)\circ a,\\
\label{cr2}\lambda \widetilde{d}_\lambda^0(b\circ a)+\widetilde{d}_\lambda^0[b,a]=\lambda \widetilde{d}_\lambda^0(a)\ast b+[b,\widetilde{d}_\lambda^0(a)]+[\widetilde{d}_\lambda^0(b),a],
\end{eqnarray}
for any $a$, $b\in V$. \\
(3) If for $(V,\circ)$, there exists $x\in V$ such that $x\circ b=kb$ for any $b\in V$ and some fixed $k\in \mathbb{C}\backslash\{0\}$, then any conformal derivation $d$ of
$R$ is of the form
$d=D+\widetilde{d^0}$ where $D$ is an inner conformal derivation and $\widetilde{d^0}$ satisfies (\ref{cr1}) and (\ref{cr2}).
\end{corollary}
\begin{proof}
Since $R=\mathbb{C}[\partial]V$ is a finitely generated and free $\mathbb{C}[\partial]$-module, we can assume that
$d_\lambda(a)=\sum_{i=0}^n\partial^i d_\lambda^i(a)$ for any $a\in V$ and some non-negative integer $n$. If $n>3$, by Theorem \ref{t2} and the condition that $V$ is a simple Novikov algebra or there exists $x\in V$ such that $x\circ y=ky$ or $y\circ x=ky$ for any $y\in V$ and $k\in \mathbb{C}\backslash\{0\}$, we can get $d_\lambda^n(b)=0$ for any $b\in V$. Therefore, for any $a\in V$, we can assume that $d_\lambda(a)=\sum_{i=0}^3\partial^i d_\lambda^i(a)$. By Theorem \ref{t2}, (1) can be directly obtained.

Next, we prove (2). According to the above discussion, we get $d_\lambda(a)=\sum_{i=0}^3\partial^i d_\lambda^i(a)$ for any $a\in V$. By Theorem \ref{t2}, (\ref{so1})-(\ref{so15}) hold. Setting $a=b=x$ in (\ref{op5}), we get
$k\lambda d_\lambda^2(x)=[d_\lambda^2(x),x]$. Since $k\neq 0$, we can obtain $d_\lambda^2(x)=0$. Letting $a=b=x$ in (\ref{op2})
we can get $[d_\lambda^3(x),x]=0$. Then it is easy to get that $x\circ d_\lambda^3(x)=0$ from (\ref{op3}). Therefore,
setting $a=b=x$ in (\ref{op4}), one can obtain $d_\lambda^3(x)\circ x=kd_\lambda^3(x)=0$. Consequently, $d_\lambda^3(x)=0$.
Letting $a=x$ in (\ref{pw2}), we get $x\circ d_\lambda^3(b)=0$ for any $b\in V$. Setting $a=x$ in (\ref{pw3}), one can directly obtain $d_\lambda^3(b)=0$ for any $b\in V$. Therefore, we get $d_\lambda(a)=\sum_{i=0}^2\partial^i d_\lambda^i(a)$ for any $a\in V$. Then by Theorem \ref{t2}, we get
\begin{eqnarray}
\label{po1}d_\lambda^2(b\circ a)=d_\lambda^2(b)\circ a,\\
\label{po2}d_\lambda^2(a\ast b)=2 d_\lambda^2(b)\circ a+d_\lambda^2(b)\ast a,\\
\label{po3}b\circ d_\lambda^2(a)+d_\lambda^2(b)\circ a+2 d_\lambda^2(b)\ast a=0,\\
\label{po4}d_\lambda^2(a)\ast b+d_\lambda^2(b)\ast a=0.
\end{eqnarray} Letting $a=x$ in (\ref{po4}), we can have
$d_\lambda^2(b)\ast x=0$. Then according to (\ref{po3}) with $a=x$, $kd_\lambda^2(b)=0$. Therefore, $d_\lambda^2(b)=0$ for any
$b\in V$. Therefore, $d_\lambda(a)=d_\lambda^0(a)+\partial d_\lambda^1(a)$ for any
$a\in V$. By Theorem \ref{t2}, we get that $d^0$ and $d^1$ satisfy (\ref{g3})-(\ref{g7}). Set $d_\lambda^1(y)=\sum_{i=0}^m\lambda^iy_i$ where $y_i\in V$ for $i\in\{0, 1, 2, \cdots, m\}$.  If
$y\circ V=V$, then there exist some $b_i$ such that $y\circ b_i=y_i$ for all $i$. It is easy to see that $(d-\text{ad}(\sum_{i=0}^m(-\partial)^ib_i)_\lambda(y)\in V[\lambda]$. Therefore, we can assume that $d^1_\lambda(y)=0$. Then letting $b=y$ in (\ref{g3}) and by $y\circ V=V$, we can get $d^1=0$. Therefore, $d= D+\widetilde{d}^0$ where $D$ is an inner conformal derivation. Obviously, $\widetilde{d}^0$ is a conformal derivation and $\widetilde{d}^0_\lambda(a)\in V[\lambda]$ for  any $a\in V$. Therefore, by (\ref{g5}) and (\ref{g7}), $\widetilde{d}^0$ satisfies (\ref{cr1}) and (\ref{cr2}).

Finally, we prove (3). Similarly, we can set $d_\lambda(a)=\sum_{i=0}^3\partial^i d_\lambda^i(a)$ for any $a\in V$.
Letting $a=x$ in (\ref{so1}) and (\ref{pw3}), we get $d_\lambda^3(b)\circ x=kd_\lambda^3(b)$ and
$2d_\lambda^3(b)\circ x+kd_\lambda^3(b)=0$. Since $k\neq 0$, we get $d_\lambda^3(b)=0$ for any $b\in V$.
By (\ref{po1}), from (\ref{po2}), we can get
\begin{eqnarray}
\label{po5}d_\lambda^2(a\circ b)=2 d_\lambda^2(b)\circ a+a\circ d_\lambda^2(b).
\end{eqnarray}
Setting $a=x$ in (\ref{po5}), we get $d_\lambda^2(b)\circ x=0$. Then letting $a=x$ in (\ref{po3}),
we obtain
\begin{eqnarray}
\label{po6}b\circ d_\lambda^2(x)+2k d_\lambda^2(b)=0,~~~~~~~ \text{for any $b\in V$.}
\end{eqnarray}
Letting $b=x$ in (\ref{po6}), we get $3kd_\lambda^2(x)=0$. Therefore, $d_\lambda^2(x)=0$. Then by (\ref{po6}), we get
$d_\lambda^2(b)=0$ for any $b\in V$. Therefore, $d_\lambda(a)=d_\lambda^0(a)+\partial d_\lambda^1(a)$ for any
$a\in V$. Then $d^0$ and $d^1$ satisfy (\ref{g3})-(\ref{g7}). Since $x\circ V=V$, with a similar discussion as that in the proof of (2), we may assume that $d^1=0$. Therefore, $d= D+\widetilde{d}^0$, where $D$ is an inner conformal derivation and $\widetilde{d}^0$ satisfies (\ref{cr1}) and (\ref{cr2}).

By now, the proof is finished.
\end{proof}

\begin{corollary}\label{coo2}
Let $R=\mathbb{C}[\partial]V$ be a quadratic Lie conformal algebra corresponding to the Novikov algebra $(V, \circ)$.
Then we have\\
(1) If $V$ is a simple Novikov algebra, then any conformal derivation $d$ of $R$ must be of the following form: $d_\lambda (a)=\sum_{i=0}^3\partial^i d_\lambda^i(a)$ for any $a\in V$, where $d_\lambda^i(a)\in V[\lambda]$ for
$i\in\{0,1,2,3\}$ satisfy (\ref{so1})-(\ref{so15}) with $[\cdot,\cdot]$ trivial.\\
(2) If for $(V,\circ)$, there exists $x\in V$ such that $b\circ x=kb$ for any $b\in V$ and some fixed $k\in \mathbb{C}\backslash\{0\}$, then any conformal derivation $d$ of
$R$ is of the form $d_\lambda(a)=d_\lambda^0(a)+\partial d_\lambda^1(a)$ for any $a\in V$, where $d_\lambda^0(a)$, $d_\lambda^1(a)\in V[\lambda]$ and they satisfy (\ref{g3}), (\ref{g4}) and \begin{gather}
\label{gg5}d_\lambda^0(b\circ a)+\lambda d_\lambda^1(b\circ a)=b\circ d_\lambda^0(a)-\lambda b\circ d_\lambda^1(a)+d_\lambda^0(b)\circ a,\\
\label{gg7}d_\lambda^0(b\circ a)=d_\lambda^0(a)\ast b-\lambda d_\lambda^1(a)\ast b,\end{gather}
for any $a$, $b\in V$.  Moreover, if $(V,\circ)$ also satisfies that there exists some $y\in V$, $y\circ V=V$, then any conformal derivation $d$ of
$R$ is of the form
$d=D+\widetilde{d^0}$ where $D$ is an inner conformal derivation and $\widetilde{d^0}$ satisfies
\begin{eqnarray}
\label{ddd1}\widetilde{d}_\lambda^0(b\circ a)=b\circ \widetilde{d}_\lambda^0(a)+\widetilde{d}_\lambda^0(b)\circ a,\\
\label{ddd2}\widetilde{d}_\lambda^0(b)\circ a=\widetilde{d}_\lambda^0(a)\circ b,
\end{eqnarray}
for all $a$, $b\in V$.\\
(3) If  for $(V,\circ)$, there exists $x\in V$ such that $x\circ b=kb$ for any $b\in V$ and some fixed $k\in \mathbb{C}\backslash\{0\}$, then any conformal derivation $d$ of
$R$ is of the form
$d=D+\widetilde{d}^0$ where $D$ is an inner conformal derivation and $\widetilde{d}^0$ satisfies
(\ref{ddd1}) and (\ref{ddd2}). Moreover, if the element $x$ also satisfies that for any $b\in V$ and the same number $k$, $b\circ x=kb$, then $\text{CDer}(R)=\text{CInn}(R)$.
\end{corollary}
\begin{proof}
For any $a\in V$, set
$d_\lambda (a)=\sum_{i=0}^{n_a}\partial^id_\lambda^i(a),$ where $d_\lambda^i(a)\in V[\lambda]$
for  $i\in \{0,1,\cdots,n_a\}$, and $n_{a}$ is a non-negative integer depending on $a$.
With the same process as in Theorem \ref{t2}, (\ref{d1}) becomes
\begin{gather}
\label{nn1}(\lambda+\partial)d_\lambda(b\circ a)+\mu d_\lambda(a\ast b)=[(d_\lambda a)_{\lambda+\mu}b]+[a_\mu(d_\lambda b)].
\end{gather}
For fixed $a$, $b$, there are four elements of $V$ under the actions of $d_\lambda$ in (\ref{nn1}). Therefore, we may assume the degrees of $\partial$ in $d_\lambda(b\circ a)$, $d_\lambda(a\ast b)$, $d_\lambda a$ and $d_\lambda b$ in (\ref{nn1}) are smaller than some non-negative integer. So, we set $d_\lambda (b\circ a)=\sum_{i=0}^{n}\partial^id_\lambda^i(b\circ a)$,
 $\cdots$ and
$d_\lambda (b)=\sum_{i=0}^{n}\partial^id_\lambda^i(b)$. Obviously, $n$ depends on $a$ and $b$.

Taking them into (\ref{nn1}), we get
\begin{gather}
(\lambda+\partial)\sum_{i=0}^n\partial^i d_\lambda^i(b\circ a)+\mu\sum_{i=0}^n \partial^i d_\lambda^i(a\ast b)\nonumber\\
=\sum_{i=0}^n(-\lambda-\mu)^i(\partial(b\circ d_\lambda^i(a))+(\lambda+\mu)d_\lambda^i(a)\ast b)\nonumber\\
\label{nn2}+\sum_{i=0}^n(\mu+\partial)^i(\partial(d_\lambda^i(b)\circ a)+\mu d_\lambda^i(b)\ast a).
\end{gather}
If $n>3$, by comparing the coefficients of $\mu^{n-1}\partial^2$ and $\mu^2\partial^{n-1}$ in (\ref{nn2}), we get
\begin{eqnarray*}
nd_\lambda^n(b)\circ a+C_n^2d_\lambda^n(b)\ast a=0,~~~C_n^2d_\lambda^n(b)\circ a+nd_\lambda^n(b)\ast a=0.
\end{eqnarray*}
Therefore, $d_\lambda^n(b)\circ a=0$ and $d_\lambda^n(b)\ast a=0$. Repeating this process, we can get
$d_\lambda^m(b)\circ a=d_\lambda^m(b)\ast a=0$ for all $n\geq m> 3$.

By the discussion above, for any $a$, $b\in V$, we get $d_\lambda^m(b)\circ a=a\circ d_\lambda^m(b)=0$ for all $m>3$. By the condition that $V$ is a simple Novikov algebra or there exists $x\in V$ such that $x\circ y=ky$ or $y\circ x=ky$ for any $y\in V$ and $k\in \mathbb{C}\backslash\{0\}$, we get $d_\lambda^m(b)=0$ for any $b\in V$ and $m>3$. Therefore, we can assume that
$d_\lambda (a)=\sum_{i=0}^{3}\partial^id_\lambda^i(a),$ for any $a\in V$. Then (1), (2) and the first part of (3) can be directly obtained from Corollary \ref{coo1}.

Finally, we prove the second part of (3). By the first part of (3), we only need to determine $\widetilde{d}^0$. For computing
$\widetilde{d}^0$, we only need to compute the operator $T: V\rightarrow V$ such that
\begin{eqnarray}
\label{crr1}T(b\circ a)=b\circ T(a)+T(b)\circ a,\\
\label{crr2}T(b)\circ a=T(a)\circ b,
\end{eqnarray}
for any $a$, $b\in V$.  Letting $a=b=x$ in (\ref{crr1}), we can directly obtain that $T(x)=0$. Setting $a=x$ in (\ref{crr2}), we get that $T(b)=0$ for any $b\in V$. Therefore, $\widetilde{d}^0=0$. Thus, all conformal derivations of $R$ are inner.
\end{proof}

\begin{remark}
Note that this corollary also holds when $V$ is infinite-dimensional. Moreover, by (3) in Corollary \ref{coo2}, all conformal derivations of the quadratic Lie conformal algebra corresponding to the Novikov algebra with a unit are inner.
\end{remark}

\begin{remark}
It should be pointed out that when $\mathfrak{g}$ is a finite-dimensional Lie algebra,
every conformal derivation $d$ of $\text{Cur}\mathfrak{g}$ is of the form $d_\lambda(a)=p(\lambda)(\partial+\lambda)a+ \widetilde{d}_\lambda(a)$, where $\widetilde{d}$ is an inner conformal derivation and $p(\lambda)\in \mathbb{C}[\lambda]$. This characterization can be referred to \cite{DK1}.
\end{remark}

Finally, we will use the above results to study conformal derivations of some specific Lie conformal algebras.
\begin{example}
It is known that $Vir$ is the quadratic Lie conformal algebra corresponding to $(V=\mathbb{C}L,\circ)$ where $L\circ L=L$. Obviously, $L$ is a unit of $V$. By (3) in Corollary \ref{coo2}, $\text{CDer}(Vir)=\text{CInn}(Vir)$. This result can also be found in \cite{DK1}.
\end{example}

\begin{example}
 The corresponding Novikov algebra of loop Virasoro conformal algebra $\mathcal{LW}=\bigoplus_{i\in \mathbb{Z}}\mathbb{C}[\partial]L_i$ introduced in Example \ref{l1}
is $V=\bigoplus_{i\in \mathbb{Z}}\mathbb{C}L_i$ with the product
$L_i\circ L_j=-L_{i+j}$ for any $i$, $j\in \mathbb{Z}$. Since $L_0\circ L_j=-L_j$ and $L_j\circ L_0=-L_j$ for any $j\in \mathbb{Z}$, by (3) in Corollary \ref{coo2}, $\text{CDer}(\mathcal{LW})=\text{CInn}(\mathcal{LW})$. This is a result in \cite{WCY}.
\end{example}

\begin{example}
Let $R(\alpha,\beta)$ be the Lie conformal algebra given in Example \ref{ex1}. Next, let us consider conformal derivations of $R(\alpha,\beta)$.

Since $L\circ L=L$, $W\circ L=W$, by Corollary \ref{coo1}, we get $n\leq 1$. Therefore, by Corollary \ref{coo1},
we only need to determine $d^0$ and $d^1$ satisfying (\ref{g3})-(\ref{g7}).
According to that $L\circ L=L $ and $L\circ W=(\alpha-1)W$, by (2) in Corollary \ref{coo1}, we can discuss it in two cases, i.e.
$\alpha\neq 1$ and $\alpha=1$.

When $\alpha\neq 1$, we have $L\circ R(\alpha,\beta)=R(\alpha,\beta)$. Then by (2) in Corollary \ref{coo1}, we only need to determine $\widetilde{d}^0$ satisfying (\ref{cr1}) and (\ref{cr2}). By (\ref{cr1}), we first study the operator $T: V\rightarrow V$ satisfying $T(b\circ a)=b\circ T(a)+T(b)\circ a$. It is easy to check that $T$ is of the form : $T(L)=0$ and $T(W)=c_2W$ for some $c_2\in \mathbb{C}$. Therefore, $\widetilde{d}_\lambda^0(L)=0$ and $\widetilde{d}_\lambda^0(W)=a(\lambda)W$ for some $a(\lambda)\in \mathbb{C}[\lambda]$.
Replacing $a$, $b$ by $L$, $W$ in (\ref{cr2}), we can easily obtain $a(\lambda)=0$. Therefore, $\widetilde{d}_\lambda^0(L)=0$, $\widetilde{d}_\lambda^0(W)=0$ and (\ref{cr2}) holds for other cases.
So, in this case, $\text{CDer}(R(\alpha,\beta))=\text{CInn}(R(\alpha,\beta)).$

Finally, we consider the case when $\alpha=1$. First, we consider $d^1$. For obtaining $d^1$, by (\ref{g3}) and (\ref{g4}), we only need to study
the operator $T: V\rightarrow V$ satisfying the following equalities:
\begin{eqnarray}
\label{er1}&&T(b\circ a)=T(b)\circ a, \\
\label{er2}&&T(a)\ast b=T(b)\ast a.
\end{eqnarray}
By a simple computation, $T$ is of the following form: $T(L)=a_1L+a_2W$ and $T(W)=a_1W$ for any $a_1$, $a_2\in \mathbb{C}$.
Therefore, $d^1$ is of the form as follows: $d_\lambda^1(L)=A(\lambda)L+B(\lambda)W$, $d_\lambda^1(W)=A(\lambda)W$ for any
$A(\lambda)$, $B(\lambda)\in \mathbb{C}[\lambda]$.
Let $D=d-\text{ad}(A(-\partial)L)$.
Then it is easy to check that $D=D^0+\partial D^1$ where $D_\lambda^i(a)\in \mathbb{C}[V]$ for any
$i\in\{0, 1\}$, $a\in V$, and $D_\lambda^1(L)=B(\lambda)W$, $D_\lambda^1(W)=0$. Next, we begin to determine
$D^0$. Set $D_\lambda^0(L)=E(\lambda)L+F(\lambda)W$. Then replacing $a$, $b$ by $L$, $L$ in (\ref{g5}), we can obtain
$E(\lambda)L+(\beta-\lambda)B(\lambda)W=0$. Therefore, $B(\lambda)=E(\lambda)=0$. So, $D^1=0$. Therefore, as the case that
$\alpha\neq 1$, we first consider the operator $T:V\rightarrow V$ satisfying $T(b\circ a)=b\circ T(a)+T(b)\circ a$. It is easy to check that $T$ is of the form : $T(L)=d_2W$ and $T(W)=e_2W$ for some $d_2$, $e_2\in \mathbb{C}$. Therefore, $D_\lambda^0(L)=F(\lambda)W$ and $D_\lambda^0(W)=G(\lambda)W$ for some $F(\lambda)$, $G(\lambda)\in \mathbb{C}[\lambda]$.
Replacing $a$, $b$ by $L$, $W$ in (\ref{cr2}), we can easily obtain $G(\lambda)=0$. Therefore, $D_\lambda^0(L)=F(\lambda)W$, $D_\lambda^0(W)=0$ and (\ref{cr2}) holds for other cases. If $F(\lambda)=\sum_{i=0}^ma_i(\lambda-\beta)^i$,
then let $\gamma(\lambda)=\sum_{i=0}^{m-1}a_{i+1}(\lambda-\beta)^i$. Then, $D^0= \text{ad}(\gamma(-\partial)W)+Q$,
where $Q_\lambda(L)=a_0W$, $Q_\lambda(W)=0$. Thus, $\text{CDer}(R(1,\beta))=\text{CInn}(R(1,\beta))\oplus M,$ where $M$ is the vector space spanned by $Q$, where $Q_\lambda(L)=W$, $Q_\lambda(W)=0$.

Therefore, by the discussion above, $\text{CDer}(R(\alpha,\beta))=\text{CInn}(R(\alpha,\beta))\oplus M,$ where $M$ is the vector space spanned by $Q$, where $Q_\lambda(L)=\delta_{\alpha,1}W$, $Q_\lambda(W)=0$.
\end{example}

\begin{example}
Let $R=\oplus_{i\geq -1}\mathbb{C}[\partial]L_i$ be a Lie conformal algebra with the following $\lambda$-bracket:
\begin{eqnarray}
[{L_i}_\lambda L_j]=((i+1)\partial+(i+j+2)\lambda)L_{i+j}, ~~~\text{for~~$i$, $j\geq -1$.}
\end{eqnarray}
It is the graded algebra of general conformal algebra $gc_1$ (see \cite{SY1}).

Obviously, it is a quadratic Lie conformal algebra corresponding to the Novikov algebra $(V=\oplus_{i\geq -1} \mathbb{C}L_i, \circ)$ where
\begin{eqnarray}
L_i \circ L_j=(j+1)L_{i+j},~~~\text{for~~$i$, $j\geq -1$.}
\end{eqnarray}
Since $L_i\circ L_0=L_i$ for all $i\geq -1$, and  $L_{-1}\circ V=V$, by (2) in Corollary \ref{coo2},
we only need to determine  $\widetilde{d}^0$ satisfying (\ref{ddd1}) and (\ref{ddd2}).
In fact, by (\ref{ddd1}) and (\ref{ddd2}),
it can be changed into find all operations $T: V\rightarrow V$ satisfying
\begin{eqnarray}\label{T1}
T(a\circ b)=T(a)\circ b+a\circ T(b),~~~T(a)\circ b=T(b)\circ a.
\end{eqnarray}

Let $T: V\rightarrow V$ be an operator satisfying (\ref{T1}). Define $T_i(L_j)=\pi_{i+j}T(L_j)$ where in general $\pi_i$ is the natural projection
from $V$ onto $L_i$. Then $T_i$ is an operator satisfying (\ref{T1}) and $T=\sum_{i\geq -1}T_i$ in
the sense that for any $x\in V$ only finitely many $T_i(x)\neq 0$. Therefore, set $T_c(L_i)=f(i)L_{i+c}$.
Replacing $T$ by $T_c$, $a$ by $L_i$ and $b$ by $L_j$ in (\ref{T1}), and comparing the coefficients of $L_{i+j+c}$,
we obtain
\begin{eqnarray}
&&\label{eq10}(j+1)f(i+j)=f(i)(j+1)+f(j)(j+c+1),\\
&&\label{eq11}f(i)(j+1)=f(j)(i+1), ~~~\text{for~~$i$, $j\geq -1$.}
\end{eqnarray}
By (\ref{eq11}), we can get
 $f(i)=A(i+1)$ for all $i$ and some $A\in \mathbb{C}$. Letting $i=0$ in (\ref{eq10}), we immediately get $c=-1$ or $f(j)=0$ for all $j$. Therefore, $T_c=0$, if $c\neq -1$ and $T_{-1}=A(i+1)L_{i-1}$. Thus, $T(L_i)=A(i+1)L_{i-1}$ for some $A\in \mathbb{C}$.
 So, $d_\lambda^0(L_i)=a(\lambda)(i+1)L_{i-1}$ for any $a(\lambda)\in \mathbb{C}[\lambda]$.
 If $a(\lambda)=\sum_{i=0}^ma_i\lambda^i$, set $b(\lambda)=\sum_{i=0}^{m-1}a_{i+1}\lambda^i$.
 Then, $d=\text{ad}(b(-\partial)L_{-1})+Q$, where $Q_\lambda(L_i)=a_0(i+1)L_{i-1}$ for any $i\geq -1$.

Therefore, by the discussion above, $\text{CDer}(R)=M\oplus \text{CInn}(R),$ where $M$ is the vector space spanned by $Q$, where $Q_\lambda(L_i)=(i+1)L_{i-1}$ for any $i\geq -1$..
\end{example}

\end{document}